\documentclass[11pt]{amsart}
\title[4-manifolds with shadow complexity zero and collapsing]
{Shadows of 4-manifolds with complexity zero and polyhedral collapsing}

\author[Hironobu Naoe]{Hironobu Naoe}
\address{Tohoku University, Sendai, 980-8578, Japan}
\email{sb3m22@math.tohoku.ac.jp}

\usepackage{amsfonts,amsmath,amssymb,amscd}
\usepackage{amsthm}
\usepackage{latexsym}
\usepackage[dvipdfmx]{graphicx}
\usepackage{psfrag}
\usepackage{color}
\usepackage{setspace}

\theoremstyle{plain}
\newtheorem{theorem}{Theorem}[section]
\newtheorem{lemma}[theorem]{Lemma}
\newtheorem{corollary}[theorem]{Corollary}
\newtheorem{proposition}[theorem]{Proposition}
\newtheorem{assertion}{Assertion}
\newtheorem{maintheorem}{Theorem}
\newtheorem{maincorollary}[maintheorem]{Corollary}

\theoremstyle{definition}
\newtheorem{definition}[theorem]{Definition}
\newtheorem{remark}[theorem]{Remark}

\newcommand{\Z}{\mathbb{Z}}
\newcommand{\R}{\mathbb{R}}

\newcommand{\Int}{\mathrm{Int}}

\newcommand{\Nbd}{\mathrm{Nbd}}

\newcommand{\gl}{\mathrm{gl}}
\newcommand{\Mobius}{M\"{o}bius }

\definecolor{akairo}{rgb}{.8,0,0}

\makeatletter

\long\def\@makecaption#1#2{
  \small
  \vskip\abovecaptionskip
  \sbox\@tempboxa{#1. #2}
  \ifdim \wd\@tempboxa >\hsize
    #1. #2\par
  \else
    \global \@minipagefalse
    \hb@xt@\hsize{\hfil\box\@tempboxa\hfil}
  \fi
  \vskip\belowcaptionskip}

\makeatother
\begin{document}
\begin{abstract}
Our purpose is to classify acyclic 4-manifolds having shadow complexity zero. 
In this paper, we focus on simple polyhedra and discuss this problem combinatorially. 
We consider a shadowed polyhedron $X$ and a simple polyhedron $X_0$ 
that is obtained by collapsing from $X$. 
Then we prove that 
there exists a canonical way to equip internal regions of $X_0$ with gleams 
so that two 4-manifolds reconstructed from $X_0$ and $X$ are diffeomorphic. 
We also show that any acyclic simple polyhedron whose singular set is a union of circles 
can collapse onto a disk. 
As a consequence of these results, 
we prove that any acyclic 4-manifold having shadow complexity zero with boundary is 
diffeomorphic to a $4$-ball. 
\end{abstract}

\maketitle
\section*{Introduction}
In 1990s Turaev introduced the notion of {\itshape shadows} of $3$- and $4$-manifolds 
with the intention of studying quantum invariants of knots and $3$-manifolds. 
A simple polyhedron $X$ is called a {\itshape shadow} of a $4$-manifold $M$ 
if $M$ collapses onto $X$ and $X$ is embedded properly and locally-flat in $M$. 
As it is well-known, $M$ can be reconstructed from $X$ and 
an appropriate decoration of its regions with some half-integers called {\itshape gleams}. 
It is called Turaev's reconstruction. 
Shadows provide many geometric properties of $3$- and $4$-manifolds. 
We refer the reader to Costantino \cite{costantino2006stein, costantino2008branched} 
for studies of Stein structures, Spin${}^c$ structures and complex structures of 4-manifolds. 
In \cite{costantino20083} 
Costantino and Thurston established the relation between shadows and 
Stein factorizations of stable maps from $3$-manifolds into $\R ^2$. 
Consequently they observed the relation between hyperbolicity of $3$-manifolds 
and their {\itshape shadow complexities}, 
which was strengthened by Ishikawa and Koda later \cite{ishikawa2014stable}. 
Here the {\itshape shadow complexity} of $M$ is defined as 
the minimum number of true vertices of a shadow of $M$. 
This notion is introduced by Costantino in \cite{costantino2006complexity}, 
in which he studied closed $4$-manifolds with shadow complexities $0$ and $1$ 
in a special case. 
In \cite{martelli2010four} Martelli completely classified the closed 
$4$-manifolds with shadow complexity $0$ in the general case. 

In this paper we study acyclic $4$-manifolds with shadow complexity zero
by observing the structure of simple polyhedra. 

To state the first theorem, we introduce a canonical way 
to equip internal regions of a subpolyhedron of a shadowed polyhedron with gleams. 
Let $(X,\gl )$ be a shadowed polyhedron, and 
$X_0$ a simple polyhedron with $X_0\subset X$. 
Let $R$ be an internal region of $X_0$. 
We observe that $Sing(X_0)\subset Sing(X)$ and that $R$ might be split by $Sing(X)$ 
into some internal regions $R_{1}, \ldots R_{n}$ of $X$ ($n\geq 1$).
Then we assign a gleam to $R$ by 
\begin{align}
\gl (R) = \sum_{i=1}^{n} \gl(R_{i}) .
\label{sum of gleams}
\end{align}

Suppose that there exist a triangulation $(K,K_0)$ of the pair $(X,X_0)$ and 
a sequence of elementary simplicial collapses from $K$ onto $K_0$.
Then we say that $X$ {\itshape collapses} $X_0$. 
Moreover, this removal is called a {\itshape polyhedral collapse} (or simply {\itshape collapse}) 
and denoted by $X\searrow X_0$. 

We denote by $M_X$ the $4$-manifold obtained from 
a shadowed polyhedron $(X,\gl)$ by Turaev's reconstruction. 
Now we state our first theorem. 

\begin{maintheorem}
\label{thm: s-collapse}
Let $(X,\gl )$ be a shadowed polyhedron, and $X_0$ a simple polyhedron collapsed from $X$. 
Assign a gleam to each internal region of $X_0$ by formula (\ref{sum of gleams}). 
Then we have $M_{X} \cong M_{X_0}$. 
\end{maintheorem}

In the second theorem, we will study acyclic simple polyhedra and their collapsibility. 
Martelli introduced a way to convert a simple polyhedron 
whose singular set is a disjoint union of circles 
to a graph in \cite{martelli2010four}. 
We will use his notations and introduce some moves on a graph that correspond to collapsings. 
\begin{maintheorem}
\label{thm: no vertex}
Any acyclic simple polyhedron whose singular set 
is a disjoint union of circles collapses onto $D^2$.
\end{maintheorem}

We have an important consequence 
of Theorem \ref{thm: s-collapse} and Theorem \ref{thm: no vertex} as follows:
\begin{maincorollary}
\label{cor: shadow complexity 0}
Every acyclic $4$-manifold with shadow complexity zero is diffeomorphic to $D^4$. 
\end{maincorollary}

This paper consists of 4 sections. 
In Section 1 we review the definitions of simple polyhedra and shadows. 
In Section 2 and in Section 3, we give the proofs of Theorem \ref{thm: s-collapse} 
and Theorem \ref{thm: no vertex} respectively. 
In Section 4, we discuss some consequences of our theorems. 

Throughout this paper, we work in smooth category unless otherwise mentioned. 

\subsection*{Acknowledgments. }
The author would like to thank his supervisor, Masaharu Ishikawa, 
for his useful comments and encouragement. 
He would also like to thank the referee for careful reading and helpful suggestions.

\section{Simple polyhedra and shadows}
\label{sec:Almost-special polyhedra and shadows}
\begin{figure}[t]
	\begin{center}
	\includegraphics[width=100mm]{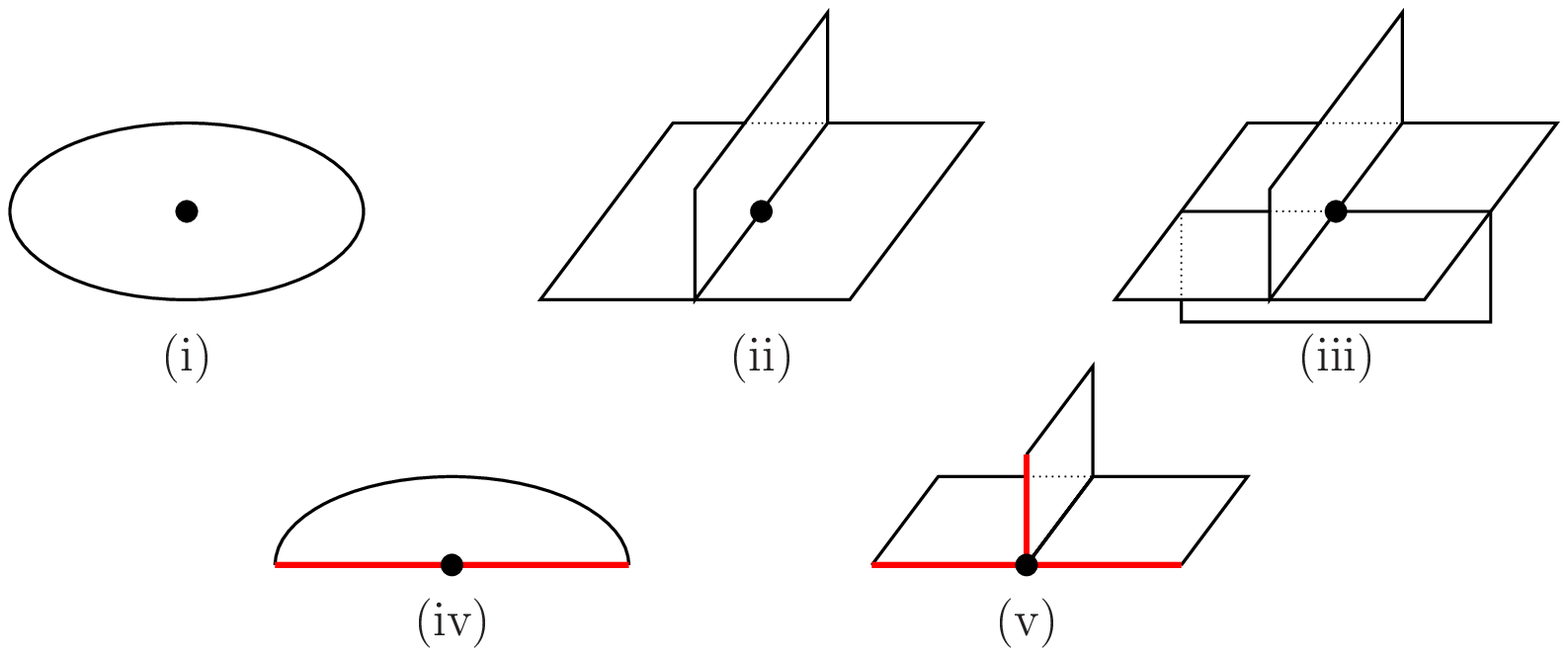}
	\caption{The local models of simple polyhedra.}
	\label{local model}
	\end{center}
\end{figure}
A compact topological space $X$ is called a {\it simple polyhedron} 
if any point $x$ of $X$ has a regular neighborhood $\Nbd (x;X)$ homeomorphic to 
one of the five local models shown in Figure \ref{local model}. 
A {\it true vertex} is a point whose regular neighborhood has a model of type (iii). 
We note that the model of type (iii) is homeomorphic to 
a cone over the complete graph $K_4$ with $4$ vertices. 
The point $(p,0)\in (K_4 \times [0,1]) / (K_4 \times \{0 \})$ is just a true vertex. 
A connected component of 
the set of points whose regular neighborhoods are of type (ii) or (v)
is called a {\it triple line}. 
The union of all true vertices and triple lines 
is called the {\it singular set} of $X$ and denoted by $Sing(X)$. 
The boundary $\partial X$ of $X$ is the set of points 
whose regular neighborhoods are of type (iv) or (v). 
Each component of $X\setminus Sing(X)$ is called a {\it region} of $X$. 
If a region $R$ contains points of type (iv) 
then $R$ is called a {\it boundary region}, 
and otherwise it is called an {\it internal region}. 
\begin{definition}
Let $M$ be a compact oriented $4$-manifold 
and let $T$ be a (possibly empty) trivalent graph in the boundary $\partial M$ of $M$.
A simple polyhedron $X$ in $M$ is called a {\it shadow} of $(M,T)$ 
if the following hold:
\begin{itemize}
\item $M$ collapses onto $X$,
\item $X$ is locally flat in $M$, that is, for each point $x$ of $X$ there exists 
a local chart $(U,\phi)$ of $M$ around $x$ such that 
$\phi (U\cap X)\subset \R^3 \subset \R^4$, and
\item $X\cap \partial M=\partial X=T$. 
\end{itemize}
\end{definition}
The following theorem by Turaev is 
very important and is called {\it Turaev's reconstruction}. 
\begin{theorem}
[Turaev \cite{turaev1994quantum}]
Let $X$ be a shadow of a $4$-manifold $M$. 
Then there exists a canonical way to equip each internal region of $X$ with a half-integer. 
Conversely, we can reconstruct $M$ uniquely from $X$ and the half-integers. 
\end{theorem}

Each half-integer in the above is called a {\itshape gleam}. 
A simple polyhedron $X$ whose internal regions are equipped with a gleam 
is called a {\it shadowed polyhedron} 
and denoted by $(X,\gl )$ (or simply $X$). 

\begin{figure}
	\begin{center}
	\includegraphics[width=36mm]{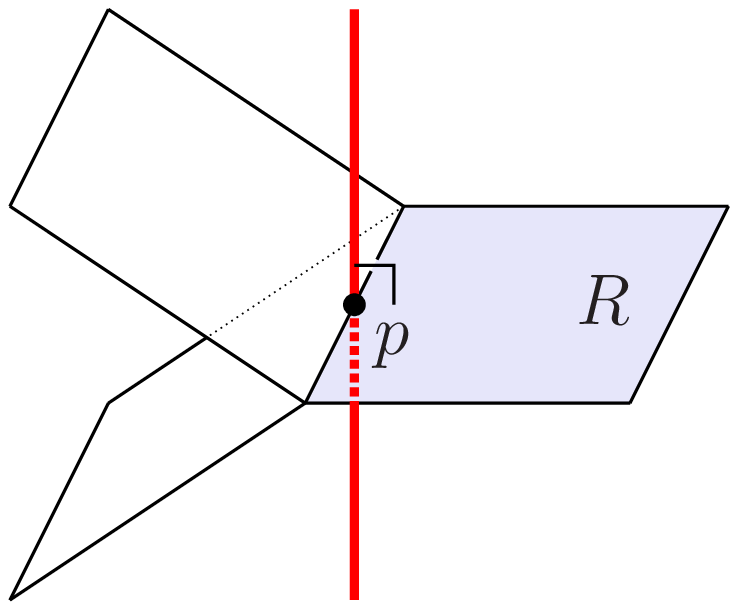}
	\caption{The fiber of the interval bundle at $p\in \partial R$.}
	\label{fiber}
	\end{center}
\end{figure}
\begin{remark}
\label{rmk; gleam}
As pointed out by Turaev \cite{turaev1994quantum}, 
a gleam generalizes the Euler number of closed surfaces embedded in oriented 4-manifolds. 
We can interpret the gleam as follows. 
Let $R$ be an intenal region of a shadow $X$ in a 4-manifold $M$ and 
let $p$ be a point of $\partial R$. 
By the locally flatness 
there exists a 3-ball $B^3$ around $p$ such that $B^3$ contains $\Nbd (p;X)$. 
We take interval which passes through $p$ and is transverse to $R$
after giving an auxiliary Euclidean metric for $B^3$ as shown in Figure \ref{fiber}. 
By taking intervals for each point of $\partial R$ continuously, 
we get an interval bundle over $\partial R$. 
We note that the interval bundle is a subbundle of the normal bundle over $\partial R$ in $M$. 
Let $R'$ be a small perturbation of $R$ 
such that $\partial R'$ is lying in the interval bundle. 
If $R'$ is generically chosen, then $R$ and $R'$ have only isolated intersections. 
Counting them with signs, we have 
\begin{align}
\label{sum of intersection}
\gl (R) = \frac{1}{2} \sharp (\partial R \cap \partial R') + \sharp (R\cap R').
\end{align}
For this formula 
we refer the reader to Carrega and Martelli \cite{Carrega2014ribbon}. 
\end{remark}

We close this section with the definition of the shadow complexity. 
\begin{definition}
Let $M$ be a compact oriented $4$-manifold having a shadow. 
We define the {\it shadow complexity} $sc(M)$ of $M$ to be 
the minimum number of true vertices of a shadow of $M$.
\end{definition}
\begin{remark}
Costantino defined the shadow complexity for ``closed'' 4-manifolds 
in \cite{costantino2006complexity}. 
A shadow of a closed 4-manifold $M$ is defined as a shadow of a 4-manifold 
to which $M$ is obtained by attaching 3- and 4-handles. 
\end{remark}
\section{Proof of Theorem \ref{thm: s-collapse}}
In this section we introduce a proposition on PL topology 
and provide a lemma 
for the proof of Theorem \ref{thm: s-collapse}. 

\begin{proposition}
\label{prop; nbd are PLhomeo}
Let $M$ be an $n$-dimensional compact PL manifold and fix a triangulation $K$ of $M$. 
Let $L_i$ be a subcomplex of a double barycentric subdivision of $K$ for $i\in \{0,1 \}$. 
If $L_0 \searrow L_1$, then $\Nbd (|L_0| ;X)$ and $\Nbd (|L_1| ;X)$ are PL-homeomorphic. 
\end{proposition}
For the proof of this proposition 
we refer the reader to \cite[Lemma 3.25, Theorem 3.26]{RourkeSanderson}.
\begin{remark}
A PL manifold has a unique smoothing in dimension $n\leq 6$ \cite{HirschMazur}. 
In our case $n=4$, 
Proposition \ref{prop; nbd are PLhomeo} with 
``PL-homeomorphic'' replaced by ``diffeomorphic'' also holds. 
\end{remark}

\begin{lemma}
\label{lem: emb shadow}
Let $(X,\gl )$ be a shadowed polyhedron, and $X_0$ a simple subpolyhedron of $X$. 
Assign a gleam to each internal region of $X_0$ by formula (\ref{sum of gleams}). 
Then we have $M_{X_0} \cong \Nbd (X_0;M_X)$. 
\end{lemma}

%

\begin{proof}
Let $K$ be the second barycentric subdivision of a given triangulation of $X_0$, 
and set 
\[
K'= K\setminus \{\tau \in K\mid \tau \cap \partial X_0 \ne \emptyset \}.
\]
Then $K'$ is a subcomplex of $K$, and $K\searrow K'$. 
We have $\Nbd (X_0;M_X) \cong \Nbd (X'_0;M_X)$ by Proposition \ref{prop; nbd are PLhomeo}, 
where $X'_0=|K'|$. 
Note that $X_0$ is proper and locally flat in $\Nbd (X'_0;M_X)$. 
Hence $X_0$ is a shadow of $\Nbd (X'_0;M_X)$. 

By Turaev's reconstruction, 
there should exist gleams for the internal regions of $X_0$ so that 
the 4-manifold reconstructed from them is diffeomorphic to $\Nbd (X'_0;M_X)$. 
It suffices to show that such gleams coincide with ones given by formula (\ref{sum of gleams}). 

Let $R$ be an internal region of $X_0$, 
and set $S=R\cap Sing(X)$. 
If $S=\emptyset$, 
the region $R$ is also an internal region for $X$. 
Hence it is obvious that their gleams coincide by Remark \ref{rmk; gleam}. 

We turn to the case $S\ne\emptyset$. 
As mentioned above, the region $R$ is split into internal regions $R_1,\ldots ,R_n$ of $X$.
\begin{figure}
	\begin{center}
	\includegraphics[width=100mm]{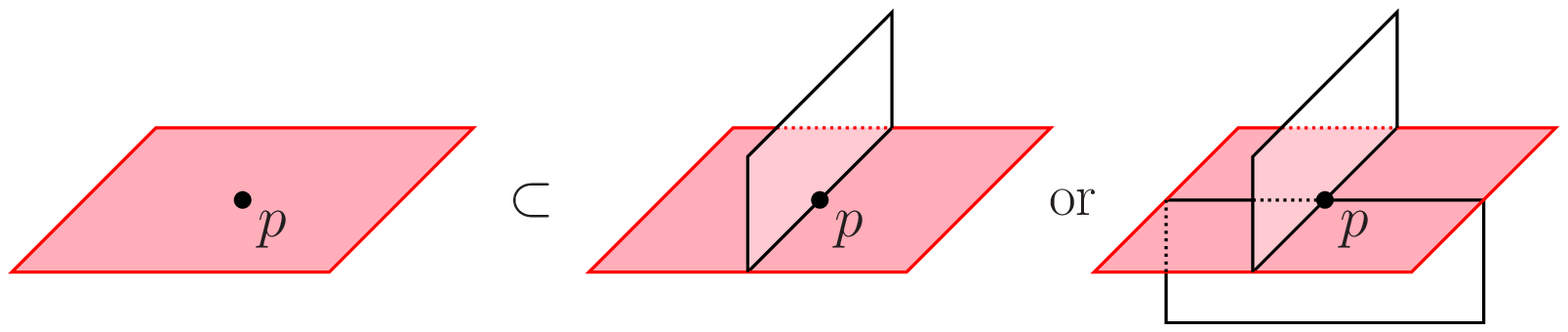}
	\caption{The leftmost picture shows $\Nbd (p;R)$.
The two right pictures are $\Nbd (p;X)$.}
	\label{smooth emb}
	\end{center}
\end{figure}
Let $p$ be a point contained in $S$. 
Then $\Nbd (p;X)$ can be described in either of the two right parts of Figure \ref{smooth emb}, 
where the colored areas indicate $\Nbd (p;R)$. 
Note that $\Nbd (p;R)=\Nbd (p;X_0)$. 
We assume that these pictures are drawn in $\R ^3$ and  
consider the regular neighborhood of $\Nbd (p;X)$. 
Carrying out Turaev's reconstruction with such 3-dimentional blocks, 
we get $M_X$ in which $R$ is smoothly embedded. 

Next we consider the interval bundle over $\partial R_i$ 
as mentioned in Remark \ref{rmk; gleam} for $i \in \{1,\ldots ,n\}$. 
By the smoothness of $R$, if $\partial R_i \cap \partial R_j \ne \emptyset$, 
the restrictions of the interval bundles of 
$\partial R_i$ and $\partial R_j$ to $\partial R_i\cap\partial R_j$ coincide. 
Hence the union of the interval bundles is regarded as 
an interval bundle over $S\cup \partial R$,
and we denote it by $L$. 
Let $S'$ be a generic small perturbation of $S$ in $L$ such that the images of the true vertices of $X$ 
do not lie in the zero section. 
Then let $R'$ ba a generic small perturbation of $R$ such that $S'\subset R'$ and 
$\partial R'$ lies in the restriction of the interval bundle $L$ to $\partial R$. 
By restricting $R'$ we get a small perturbation $R'_i$ of $R_i$ as in Remark \ref{rmk; gleam}
for $i \in \{1,\ldots ,n\}$. 
%
%
%
Note that each point $p\in S$ except for a true vertex 
is sandwiched between $R_i$ and $R_j$ for some $i,j\in\{ 1,\ldots ,n\}$, 
in other words, the point $p$ belongs to both $\partial R_i$ and $\partial R_j$. 
By formula (\ref{sum of intersection}), we have the following: 
\begin{align*}
\gl (R) =& \frac{1}{2} \sharp \left(\partial R \cap \partial R'\right) + \sharp \left(R\cap R'\right) \\
=& \frac{1}{2} \sum_{i=1}^{n} \sharp 
\left( \left(\partial R_i \cap \partial R\right) \cap \left(\partial R'_i \cap \partial R'\right) \right)
+ \sharp \left( S\cap S'\right) 
+\sum_{i=1}^{n} \sharp \left( R_i \cap R'_i \right) \\
=& \frac{1}{2} \sum_{i=1}^{n} \sharp \left( \partial R_i \cap \partial R'_i\right) 
+\sum_{i=1}^{n} \sharp \left( R_i \cap R'_i\right) \\
=& \sum_{i=1}^{n} \gl(R_{i}),
\end{align*}
and the proof is completed. 
\end{proof}

\begin{proof}
[Proof of Theorem \ref{thm: s-collapse}]
There exist diffeomorphisms 
\[
M_X \cong \Nbd (X;M_X) \cong \Nbd (X_0;M_X) \cong M_{X_0}
\] 
by Turaev's reconstruction, Proposition \ref{prop; nbd are PLhomeo} and 
Lemma \ref{lem: emb shadow}.
\end{proof}

\section{Proof of Theorem \ref{thm: no vertex}}
In this section we study simple polyhedra and its graph notation introduced by Martelli, 
and give the proof of Theorem \ref{thm: no vertex}. 

\subsection{Basic lemmas for simple polyhedra}
We first introduce convenient lemmas. 
We note that the first one was shown by Ikeda in \cite[Lemma 12]{ikeda1971acyclic}. 
\begin{lemma}
[Ikeda \cite{ikeda1971acyclic}]
\label{lem:no genus}
Any region of a simple polyhedron $X$ is orientable and 
has no genus if $H_1 (X;\Z )=0$.
\end{lemma}
The proof of Lemma \ref{lem:no genus} is given by using Mayer Vietoris exact sequence. 
For any region $R$ of $X$, he considered a closed surface $\tilde{R}$ 
obtained from $R$ by capping off the all boundary components of $R$ by disks
and constructed a new simple polyhedron that contains $\tilde{R}$. 
He checked that $H_1(\tilde{R};\Z )$ must vanish. 
\begin{lemma}
\label{lem:scc splits P}
Let $X$ be an acyclic simple polyhedron 
and $\gamma$ be a simple closed curve in $X\setminus Sing(X)$. 
Then $\gamma$ splits $X$ into two parts such that 
one of them is acyclic and the other is a homology-$S^1$.
Moreover the first homology of the latter is generated by $\gamma$.
\end{lemma}
\begin{proof}
Since any region of $X$ is orientable by Lemma \ref{lem:no genus}, 
$\Nbd (\gamma ;X)$ is homeomorphic to an annulus. 
Set $\Gamma = \Nbd (\gamma ;X)$ and $X'=X\setminus \Int \Nbd (\gamma ;X) $. 
Using the Mayer Vietoris exact sequence for the decomposition $X=\Gamma \cup X'$, 
we have the isomorphism 
$H_{q}(\Gamma ;\Z )\oplus H_{q}(X';\Z )\cong H_{q} (\Gamma \cap X';\Z )$ for $q\geq 1$. 
Hence $H_{q}(X';\Z )=0$ for $q\geq 2$ and 
$H_{1}(X';\Z )\cong \Z$. 
Moreover $\chi (X')=1$ holds from the following equality 
\[
\chi (X) = \chi (\Gamma ) +\chi (X') - \chi (\Gamma \cap X'). 
\]
Hence $\text{rank}H_{0}(X';\Z ) =2 $, that is, 
$X'$ has two connected components: 
one of them is acyclic and the other is a homology-$S^1$.
Let $X_0$ be the acyclic connected component of $X'$ and 
$X_1$ the other component. 
We regard $X$ as the union $X_0 \cup X_1$. 
From the Mayer Vietoris exact sequence for this decomposition, 
it follows that $\gamma$ generates $H_1(X_1 ;\Z)$.
\end{proof}

\subsection{Martelli's graph encoding a simple polyhedron}
Let $X$ be a simple polyhedron whose singular set is a disjoint union of circles. 
In \cite{martelli2010four}, Martelli introduced a graph encoded from $X$ 
and classified closed $4$-manifolds with shadow complexity zero. 

\begin{figure}[t]
\begin{center}
	\includegraphics[width=6.5cm]{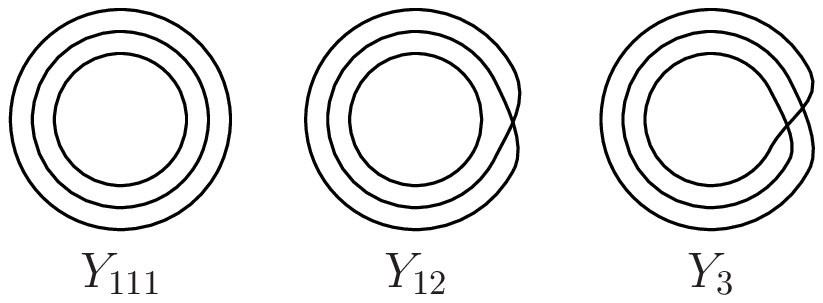}
	\caption{The three types of $Y$-bundle over $S^1$. 
These pictures indicate the only boundaries of them. }
	\label{nbd_of_S1}
\end{center}
\end{figure}
A regular neighborhood of $S^1 \subset Sing(X)$ has a structure of $Y$-bundle over $S^1$, 
where $Y$ is a cone of three points. 
There are three topological types $Y_{111}, Y_{12}$ and $Y_{3}$, 
and they are shown in Figure \ref{nbd_of_S1}. 
Each connected component of $X\setminus \Nbd (Sing(X);X)$ is a surface 
homeomorphic to a region of $X$. 
Any surface is decomposed into disks, pairs of pants and \Mobius strips. 
Hence we have the following. 
\begin{proposition}
[Martelli \cite{martelli2010four}]
\label{prop: decomposition of P}
Any simple polyhedron whose singular set is a disjoint union of circles 
is decomposed along simple closed curves into pieces homeomorphic to 
$D^2$, a pair of pants, the \Mobius strip, $Y_{111}$, $Y_{12}$ or $Y_{3}$. 
\end{proposition}
A decomposition of $X$ as in Proposition \ref{prop: decomposition of P} 
provides a graph $G$ consisting of some edges and vertices 
(B), (D), (P), (2), (111), (12) or (3) as in Figure \ref{vertices}. 
\begin{figure}[t]
\begin{center}
	\includegraphics[width=70mm]{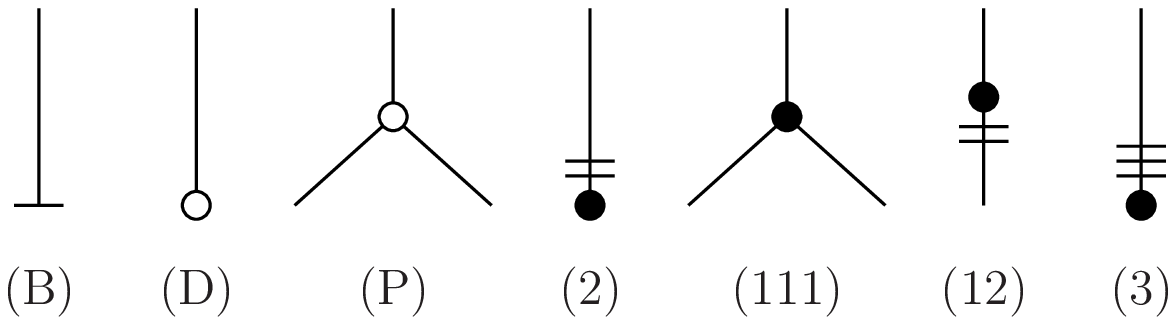}
	\caption{A simple polyhedron without true vertices is encoded by a graph having these vertices.}
	\label{vertices}
\end{center}
\end{figure}
The vertices of type (D), (P), (2), (111), (12) and (3) denote some portions homeomorphic to 
$D^2$, a pair of pants, the \Mobius strip, $Y_{111}$, $Y_{12}$ and $Y_{3}$ respectively. 
A vertex of type (B) denotes a boundary component of $X$. 
Note that each edge encodes a simple closed curve along which $X$ decomposes 
except the edges adjoining a vertex of type (B). 
We also note that we distinguish the two edges adjoining the vertex of type (12): 
the edge marked with two lines corresponds to a simple closed curve winding twice along 
the circle in $Sing(X)$. 

As Martelli said, 
we can uniquely reconstruct the simple polyhedron $X$ 
from a pair consisting of a graph $G$ and a map $\beta :H_1 (G;\Z _2)\to \Z _2$ . 
It is necessary to choose homeomorphisms that glues polyhedral pieces at each edge of $G$ 
since there are two self-homeomorphisms of $S^1$, orientation-preserving and -reversing, 
up to isotopy. 
It is encoded by a map from $H_1 (G;\Z _2)$ to $\Z _2$. 

\begin{figure}[t]
\begin{center}
	\includegraphics[width=11cm]{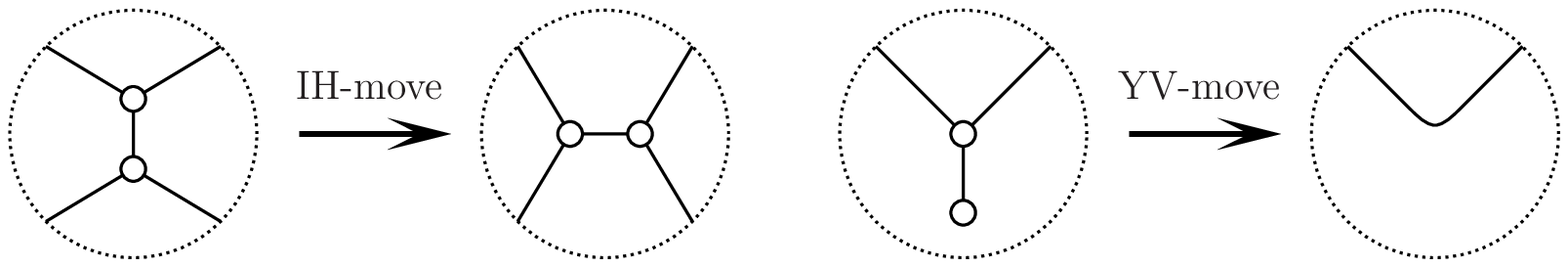}
	\caption{The move in the left part of the figure, called {\itshape IH-move}, 
corresponds to A-move for pants-decompositions in \cite{hatcher1999pants}. 
The move in the right part of the figure, called {\itshape YV-move}, 
means that an annulus plays a role of connecting two polyhedral pieces.} 
	\label{IH- and YV- moves}
\end{center}
\end{figure}

The graph $G$ that describes $X$ is not unique 
since a surface decomposes into disks, pairs of pants and \Mobius strips in several ways. 
There are some local moves as two examples shown in Figure \ref{IH- and YV- moves} 
that do not change the topological type of the polyhedron. 
We call the two moves in Figure \ref{IH- and YV- moves} 
{\itshape IH-move} and {\itshape YV-move}. 

\begin{definition}
If there exists an edge connecting two vertices $u$ and $v$ of a graph, 
then $u$ and $v$ are said to be {\itshape adjacent}. 
\end{definition}
\subsection{Acyclic case.}

Let $X$ be an acyclic simple polyhedron whose singular set consists of circles. 
We note that we only need to consider the case $\partial X\ne \emptyset$ 
since there is no acyclic closed simple polyhedron 
without true vertices \cite[Theorem 1]{ikeda1971acyclic}.

Let $G$ be a graph obtained from $X$. 
Our goal is to transform $G$ into a 1-valent graph whose vertices are 
type (B) and type (D) as shown in Figure \ref{goal graph}. 
It is obvious that this graph corresponds to the polyhedron $D^2$. 

\begin{assertion}
\label{assrt: no sigma}
There is no embedded $Y_{3}$ in any acyclic simple polyhedron. 
\end{assertion}
\begin{proof}
Assume that there exists $Y_{3}$ in an acyclic simple polyhedron $X$. 
By Lemma \ref{lem:scc splits P}, 
a simple closed curve $\partial Y_{3}$ 
splits $X$ into $Y_{3}$ and an acyclic subpolyhedron 
but does not generate $H_1 (Y_{3};\Z)$. 
It is a contradiction. 
\end{proof}
\begin{assertion}
\label{assrt: tree}
The graph $G$ is a tree. 
\end{assertion}
\begin{proof}
This follows readily from Lemma \ref{lem:scc splits P}. 
\end{proof}

\begin{figure}
\begin{tabular}{ccc}	
\begin{minipage}{0.45\hsize}
	\begin{center}
	 \includegraphics[width=14mm]{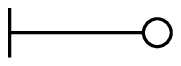}
	 \caption{This graph encodes $D^2$.}
	 \label{goal graph}
	\end{center}
\end{minipage}
\begin{minipage}{0.05\hsize}
$ $
\end{minipage}
\begin{minipage}{0.45\hsize}
	\begin{center}
	 \includegraphics[width=25mm]{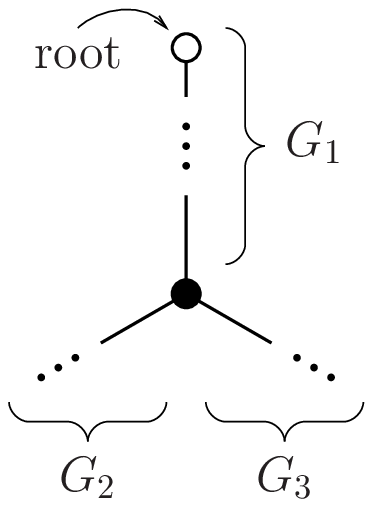}
	 \caption{The central vertex of type (111) is $v_0$. 
The subgraph $G_1$ contains the root of $G$. }
	 \label{around id}
	\end{center}
\end{minipage}
\end{tabular}
\end{figure}

If $X$ is decomposed into some pieces as in Proposition \ref{prop: decomposition of P}, 
there must be at least one piece homeomorphic to $D^2$ 
by applying iteratively Lemma \ref{lem:scc splits P}. 
Choose one of such pieces and let the corresponding vertex of type (D) be the root of $G$. 

We assume that $G$ has at least one vertex of type (111). 
We transform $G$ into a tree having no vertex of type (111). 

Consider the farthest vertex of type (111) from the root. 
Let us denote it by $v_0$. 
If we remove the corresponding piece $Y_{111}$ from $X$, 
it is decomposed into three subpolyhedra. 
Let $X_1$ be one of them such that it contains $D^2$ corresponding to the root 
and $X_2$ and $X_3$ be the remaining two subpolyhedra. 
Let $G_1$, $G_2$ and $G_3$ be the subgraphs of $G$ 
corresponding to $X_1$, $X_2$ and $X_3$ respectively as shown in Figure \ref{around id}. 
By Lemma \ref{lem:scc splits P}, at least one of $X_2$ and $X_3$ is homologically $S^1$. 
Assume that $X_2$ is so. 
Let $\gamma$ be the simple closed curve that cuts $X_2$ off from $X$.
By Lemma \ref{lem:scc splits P}, $\gamma$ generates $H_1(X_2;\Z)$. 
If $X_2$ has no boundary except $\gamma$, 
the simple polyhedron obtained from $X_2$ by capping off the boundary component by a disk 
is closed, acyclic and without true vertices, contrary to \cite[Theorem 1]{ikeda1971acyclic}. 
Hence $X_2$ has some boundary components other than $\gamma$. 
In other words, there exists a vertex of type (B) in $G_2$. 

\begin{assertion}
\label{assrt: (12)-(P)-(B)}
Let $v_1$ be a vertex of type (P) 
such that it is adjacent to a vertex of type (B). 
If $v_1$ is adjacent to a vertex of type (12), 
the edge between them is marked with two lines. 
\end{assertion}
\begin{proof}
Assume that there is an edge, denoted by $e$, adjoining $v_1$ and a vertex of type (12) 
such that the edge is not marked with two lines. 
Let $v_2$ be the vertex of type (12). 
Along the simple closed curve corresponding to $e$, 
$X$ is decomposed into two subpolyhedra: 
one contains a region $R$ corresponding to $v_1$ and 
another contains $Y_{12}$ corresponding to $v_2$. 
Since $R$ has two boundary components, 
the former subpolyhedron has a 1-cycle. 
By Lemma \ref{lem:scc splits P}, the latter subpolyhedron should be acyclic. 
However the subpolyhedron collapses so that it contains a \Mobius strip in a region. 
This contradicts Lemma \ref{lem:no genus}. 
\end{proof}

\begin{figure}
	\begin{center}
	\includegraphics[width=9cm]{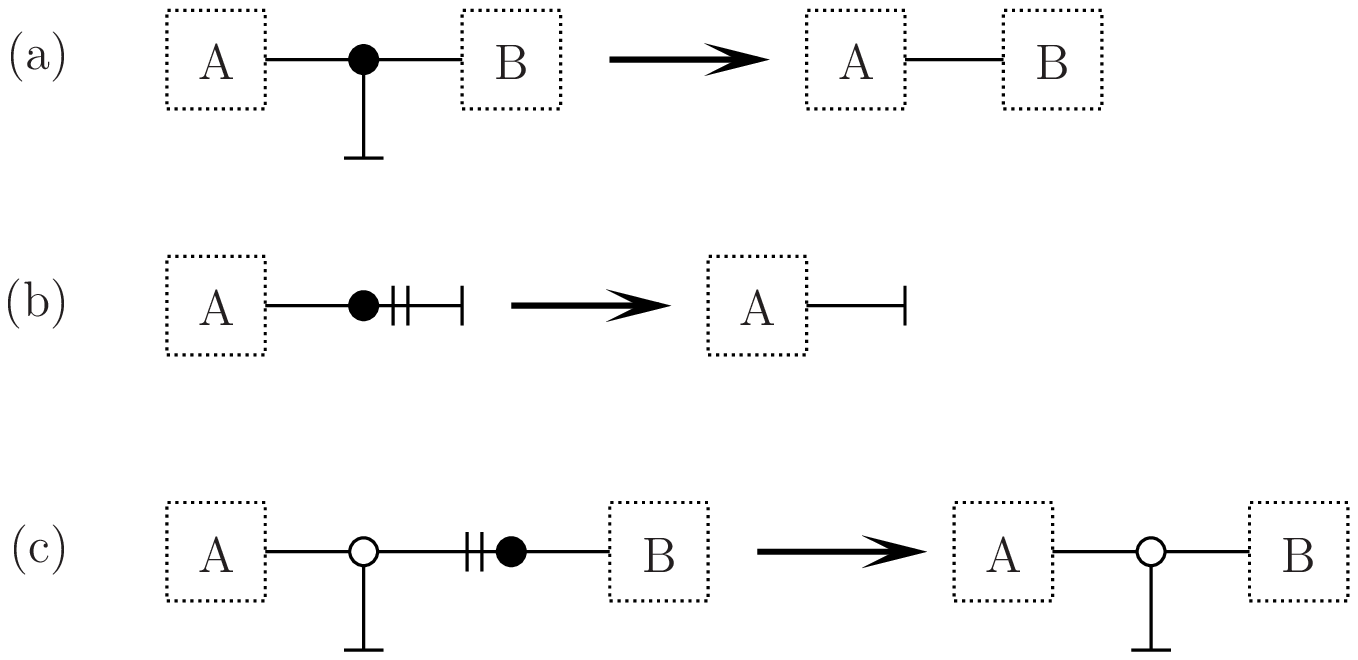}
	\caption{Three moves of graphs in Assertion \ref{assrt: moves of collapsing}.}
	\label{moves of collapsing(graph)}
	\end{center}
	\begin{center}
	\includegraphics[width=13cm]{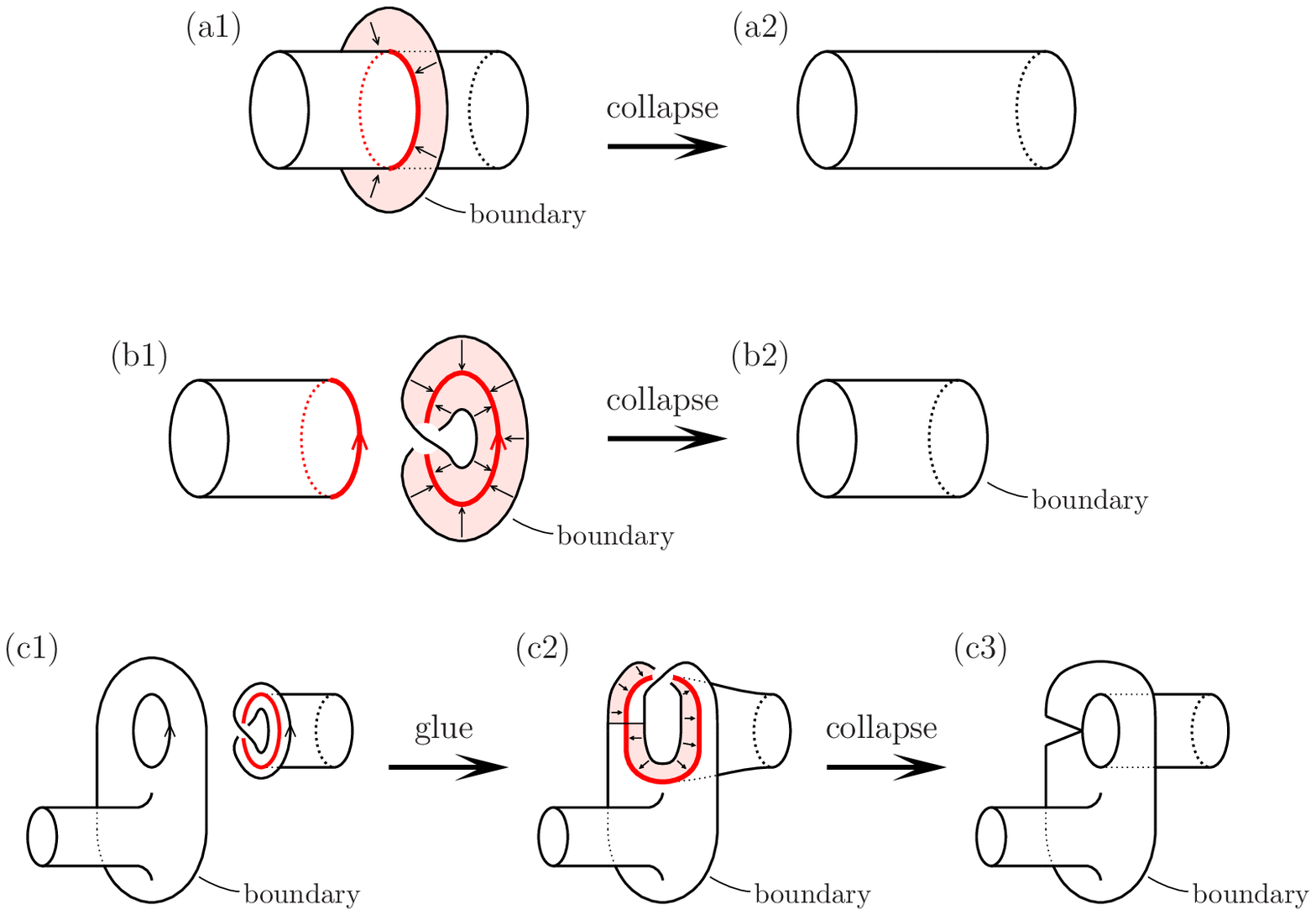}
	\caption{Collapses realizing the moves of graph 
in Figure \ref{moves of collapsing(graph)}.}
	\label{moves of collapsing(poly)}
	\end{center}
\end{figure}
\begin{assertion}
\label{assrt: moves of collapsing}
The moves of $G$ described in Figure \ref{moves of collapsing(graph)} (a), (b) and (c) are 
realized by collapsings of $X$.
\end{assertion}
\begin{proof}
(a) (resp. (b)) The corresponding part of $X$ is shown in 
Figure \ref{moves of collapsing(poly)} (a1) (resp. (b1)). 
It can collapse along the boundary component in the direction of the arrows described 
in the figure, 
and the resulting polyhedron is shown in 
Figure \ref{moves of collapsing(poly)} (a2) (resp. (b2)). 

(c) Figure \ref{moves of collapsing(poly)} (c1) shows the corresponding part of $X$. 
After we glue the two pieces, a pair of pants and $Y_{12}$ 
as in Figure \ref{moves of collapsing(poly)} (c1), 
it turns out that $X$ can be described as in Figure \ref{moves of collapsing(poly)} (c2). 
Let $X$ collapse along a part of the boundary component as indicated by the arrows. 
Then the resulting polyhedron is shown in Figure \ref{moves of collapsing(poly)} (c3). 
\end{proof}

\begin{assertion}
\label{assrt: (111) disappear}
The graph $G$ can change into a subgraph that does not contain $v_0$ and $G_2$ 
only by the moves (a),(b),(c) in Figure \ref{moves of collapsing(graph)}, 
the IH-move and the YV-move. 
\end{assertion}
\begin{proof}
Let $v$ be a vertex of type (B) in $G_2$. 
Then $v$ is adjacent to a unique vertex $v'$ of type (111), (12) or (P). 

If $v'$ is of type (111), that is $v'=v_0$, 
we apply the move (a) and the proof is completed. 

We consider the case where $v'$ is type (12) or (P). 
If $v'$ is of type (12), 
the edge between $v$ and $v'$ must be marked with two lines 
as shown in Figure \ref{possible_cases} (1) 
by the same reason of Assertion \ref{assrt: (12)-(P)-(B)}. 
If $v'$ is of type (P), 
the two vertices adjacent to $v'$ other than $v$ 
can not be of type (B) by the acyclicness of $X$. 
Therefore the possible cases are as shown in Figure \ref{possible_cases}. 

Set $H_0 = G_2$. 
In each case of (1)--(7) in Figure \ref{possible_cases}, 
we apply some moves to $H_k$ as follows and denote the resulting graph by $H_{k+1}$ 
($k=0,1,\ldots$). 
\begin{figure}
	\begin{center}
	\includegraphics[width=9cm]{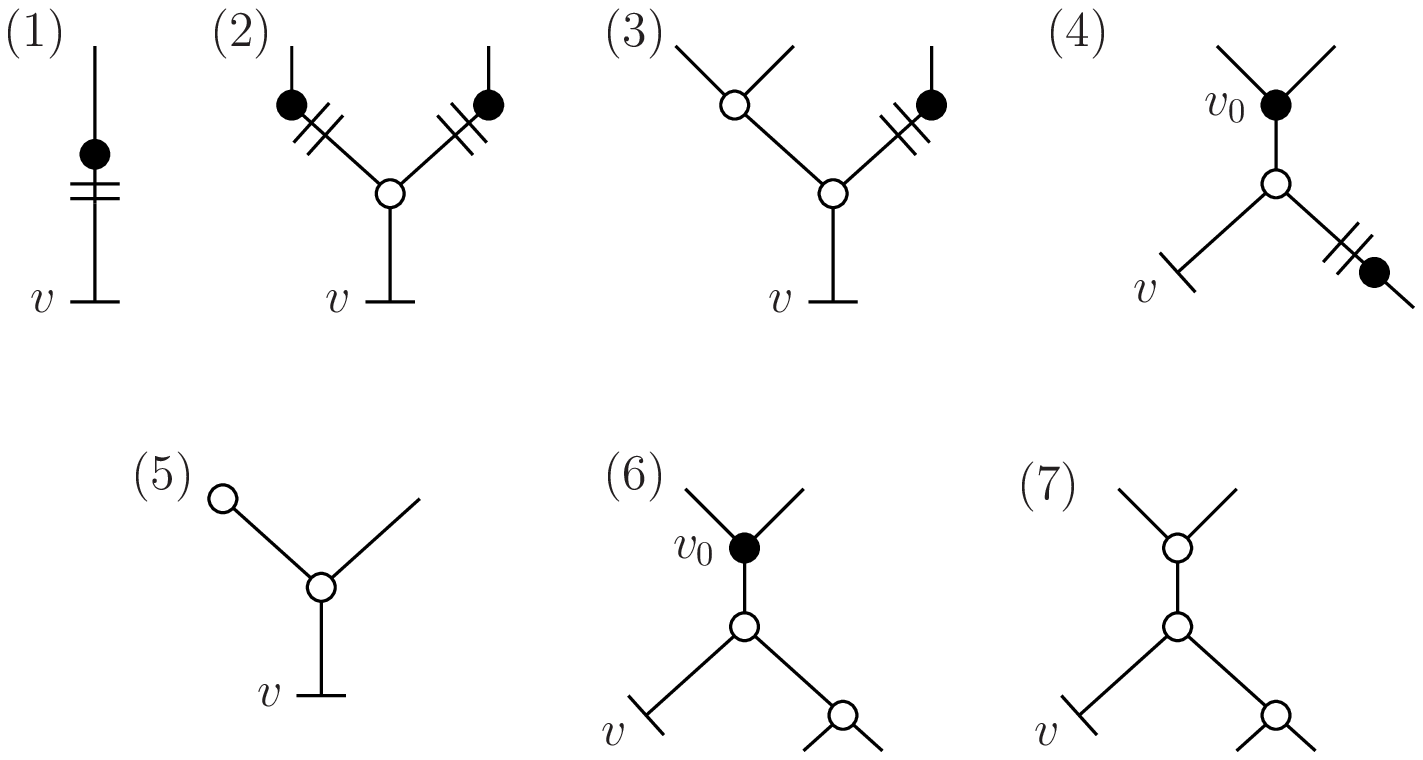}
	\caption{If a vertex of type (B) is adjacent to one of type (12) or (P), 
the graph must have one of these structures locally. }
	\label{possible_cases}
	\end{center}
\end{figure}
\begin{itemize}
\item[(1)] 
In this case we apply the move (b). 
Then the number of vertices in $H_{k+1}$ is less than the one in $H_k$. 
\item[(2)--(4)]
In these cases we apply the move (c). 
Then the number of vertices in $H_{k+1}$ is less than the one in $H_k$. 
\item[(5)]
In this case we apply YV-move. 
Then the number of vertices in $H_{k+1}$ is less than the one in $H_k$. 
\item[(6),(7)]
Let $v'_k$ be the vertex adjacent to $v$ in $H_k$. 
We define a set $V_k$ of vertices of type (P) as follows: 
$u\in V_k$ if and only if there is a path between $v'_k$ and $u$ in $H_k$ 
such that it contains vertices only of type (P). 
Then there exists at least one vertex in $V_k$ 
which is adjacent to a vertex different from $v$ and of type other than (P) or (111) 
since $G$ is a tree. 
Choose such a vertex $u$ in $V_k$ and a path between $v'_k$ and $u$ as above. 
We apply IH-moves along the path, 
and then the resulting graph $H_{k+1}$ is as in one of the cases (1)--(5). 
\end{itemize}
In all cases (1)--(7) we can decrease the number of vertices in $H_k$ but this is finite. 
Hence it comes down to the case where $v$ is adjacent to one of type (111), 
and the proof is completed. 
\end{proof}

We iterate Assertion \ref{assrt: (111) disappear} 
until all the vertices of type (111) in $G$ disappear 
and denote the resulting graph by $G'$. 
There still exists a vertex of type (B) in $G'$ by \cite[Theorem 1]{ikeda1971acyclic}. 
We denote it by $v$. 
The vertex adjacent to $v$ is of type (D), (12) or (P). 
If it is of type (D), $G'$ is as in Figure \ref{goal graph}, 
which corresponds to $D^2$. 
If the vertex adjacent to $v$ is of type (12) or (P), 
the possible cases are (1)--(3),(5),(7) in Figure \ref{possible_cases}. 
As in the proof of Assertion \ref{assrt: (111) disappear} 
we apply some moves to $G'$ and finally obtain a graph 
such that $v$ is adjacent to one of type (D). 
This completes the proof of Theorem \ref{thm: no vertex}. 

\section{Applications}


In this section, we disscuss applications of our results. 
\begin{proof}
[Proof of Corollary \ref{cor: shadow complexity 0}]
Let $M$ be an acyclic $4$-manifold with shadow complexity zero, 
and let $X$ be a shadow of $M$ without true vertices. 
Then each connected component of $Sing(X)$ is $S^1$ or a closed interval. 
If $Sing(X)$ consists of only circles, 
then $X$ collapses onto $D^2$ by Theorem \ref{thm: no vertex}. 
By Theorem \ref{thm: s-collapse}, a disk $D^2$ is a shadow of $M_X$, 
and then $M_X$ is diffeomorphic to $D^4$. 

Assume that $Sing(X)$ has a closed interval component. 
We decompose $X$ into subpolyhedra $X_1,\ldots ,X_n$ 
along all closed interval components of $Sing(X)$. 
From Turaev's reconstruction \cite{turaev1994quantum}, 
the boundary connected sum $M_{X_1} \natural \cdots \natural M_{X_n}$ and $M_{X}$ are diffeomorphic. 
For $1\leq i \leq n$, the 4-manifold $M_{X_i}$ is diffeomorphic to $D^4$ as mentioned above
since $Sing(X_i)$ consists of only circles. 
It follows that $M_X \cong D^4$. 
\end{proof}

We have another application to the study of 3-manifolds.
Let $N$ be a closed connected $3$-manifold. 
A shadow of $N$ is defined as a shadow of a $4$-manifold whose boundary is $N$. 
Costantino and Thurston indicated in \cite{costantino20083} 
that the Stein factorization of a stable map on $N$ to $\R ^2$ can be seen as a shadow of $N$ 
with a certain modification if necessary. 
The notion of complexity of stable maps was introduced 
by Ishikawa and Koda in \cite{ishikawa2014stable}. 
Especially, the complexity of a stable map is zero 
if and only if the map has no singular fiber of type $\text{II}^2$ and type $\text{II}^3$. 
See \cite{saeki2004,ishikawa2014stable} for the precise definitions. 
These observations and our results immediately yield the following: 
\begin{corollary}
\label{rmk: stein fac}
Let $N$ be a closed connected $3$-manifold. 
Then $N$ admits a stable map with complexity zero and 
acyclic Stein factorization 
if and only if $N$ is homeomorphic to $S^3$. 
\end{corollary}

\end{document}